\documentclass[12pt,makeidx]{amsart}

\usepackage{amssymb,mathrsfs, amsmath, amsthm}
\usepackage{url,titletoc}
\usepackage{enumerate,mathrsfs}

\allowdisplaybreaks

\usepackage[usenames,dvipsnames]{xcolor}
\usepackage[pagebackref,colorlinks,linkcolor=BrickRed,citecolor=OliveGreen,urlcolor=black,hypertexnames=true]{hyperref}

\newtheorem{theorem}{Theorem}[section]
\newtheorem{cor}[theorem]{Corollary}
\newtheorem{lemma}[theorem]{Lemma}

\newtheorem{prop}[theorem]{Proposition}

\newtheorem{remark}[theorem]{Remark}

\newtheorem{lem}[theorem]{Lemma}

\def\C{\mathbb {C}}

\def\cD{\mathcal D}

\def\cB{\mathcal B}

\def\IHA{I_H\otimes A}
\def\sW{\mathscr W}
\def\tR{\widetilde{R}}
\def\cW{\mathcal W}

\def\cR{\mathcal R}

\def\cA{\mathcal A}

\def\La{\Lambda}

\def\cG{\mathcal{G}}
\def\wtC{\widetilde{C}}
\def\JJ{J}
\def\GG{f}
\def\interior{\operatorname{int}}
\def\cU{\mathcal{U}}
\def\cV{\mathcal{V}}
\def\rr{q}
\def\pp{p}
\def\qq{q}

\newcommand{\FG}{G}
\newcommand{\spann}{\operatorname{span}}
\newcommand{\SigmaT}{T}

\DeclareMathOperator{\rg}{rg}

\numberwithin{equation}{section}

\newcommand{\df}[1]{{\it{#1}}{\index{#1}}}

\makeindex

\usepackage{upgreek}
\usepackage{cmap}

\title[Mapping spectraballs to free spectrahedra]{Free bianalytic maps between spectrahedra and spectraballs in a generic setting}

\author[M.L. Augat]{Meric Augat}
\address{Meric Augat, Department of Mathematics\\
  University of Florida\\ Gainesville 
   }
   \email{mlaugat@math.ufl.edu}

\author[J.W. Helton]{J. William Helton${}^1$}
\address{J. William Helton, Department of Mathematics\\
  University of California \\
  San Diego}
\email{helton@math.ucsd.edu}
\thanks{${}^1$Research supported by the NSF grant
DMS-1500835.}

\author[I. Klep]{Igor Klep${}^{2}$}
\address{Igor Klep, Department of Mathematics, 
The University of Auckland, New Zealand}
\email{igor.klep@auckland.ac.nz}
\thanks{${}^2$Supported by the Marsden Fund Council of the Royal Society of New Zealand. Partially supported by the Slovenian Research Agency grants J1-8132, N1-0057 and P1-0222.}

\author[S. McCullough]{Scott McCullough${}^3$}
\address{Scott McCullough, Department of Mathematics\\
  University of Florida\\ Gainesville 
   }
   \email{sam@math.ufl.edu}
\thanks{${}^3$Research supported by NSF grants DMS-1361501 and DMS-1764231} 

\dedicatory{This paper, which would not exist without techniques he pioneered,
 is dedicated to Joe Ball on the occasion of his 70th birthday.}

\date{\today}

\subjclass[2010]{47L25, 32H02, 13J30    (Primary); 14P10, 52A05, 46L07 (Secondary)}

\keywords{bianalytic map, birational map, 
 linear matrix inequality (LMI),
spectrahedron, convex set, Positivstellensatz, free analysis, real algebraic geometry}

\begin{document}


\numberwithin{equation}{section}

\begin{abstract}
Given a tuple  $E=(E_1,\dots,E_g)$ of $d\times d$ matrices, the collection $\cB_E$ of those tuples of matrices  $X=(X_1,\dots,X_g)$ (of the same size)
such that $\| \sum E_j\otimes X_j\|\le 1$ is a spectraball. Likewise, given a tuple $B=(B_1,\dots,B_g)$ of $e\times e$ matrices
the collection $\cD_B$ of tuples of matrices $X=(X_1,\dots,X_g)$ (of the same size) such that  $I + \sum B_j\otimes X_j +\sum B_j^* \otimes X_j^*\succeq 0$
is a free spectrahedron. Assuming $E$ and $B$ are irreducible, plus an additional mild hypothesis, there is a free bianalytic map $p:\cB_E\to \cD_B$
normalized by $p(0)=0$ and $p^\prime(0)=I$ if and only if $\cB_E=\cB_B$ and $B$ spans an algebra. Moreover $p$ is unique, rational and has an elegant algebraic
representation. 
\end{abstract} 

\maketitle

\setcounter{tocdepth}{3}
\contentsmargin{2.55em} 
\dottedcontents{section}[3.8em]{}{2.3em}{.4pc} 
\dottedcontents{subsection}[6.1em]{}{3.2em}{.4pc}
\dottedcontents{subsubsection}[8.4em]{}{4.1em}{.4pc}

\section{Introduction}
In this article we continue our investigation of free bianalytic mappings 
between matrix convex domains. The results in this article stand on the bedrock of
the noncommutative state space methods
introduced to the operator theory community by Joe and his collaborators and they are inseparable
from the profound influence of
Joe's work in function theoretic operator theory and free analysis.

 Fix $g$ a positive integer. Given
a positive integer $n$, let $M_n(\C)^g$ denote the $g$-tuples 
 $X=(X_1,\dots,X_g)$ of $n\times n$ matrices with entries from $\C$.
Given $A\in M_d(\C)^g$, the set $\cD_A(1)$ consisting  of   $x\in\C^g$ such that
\begin{equation*}
 L_A(x) = I+\sum A_jx_j +\sum A_j^* x_j^* \succeq 0
\end{equation*}
 is a \df{spectrahedron}. Here $T\succeq 0$ indicates the selfadjoint matrix $T$ is positive semidefinite.
Spectrahedra are basic objects in a number of areas of mathematics;
e.g.~semidefinite programming, convex optimization and in real algebraic geometry \cite{BPR13}. 
They also figure prominently in determinantal representations \cite{Bra11,GK-VVW,NT12,Vin93}, the solution of the Lax conjecture \cite{HV07},
in the solution of the Kadison-Singer paving conjecture \cite{MSS15},
 and in systems engineering \cite{BGFB94, Skelton}.

For $X\in M_n(\C)^g$ and still with $A\in M_d(\C)^g$, let 
\[
\Lambda_A(X) = \sum A_j \otimes X_j
\]
and 
\[
 L_A(X) = I+\Lambda_A(X) +\Lambda_A(X)^* = I+\sum A_j \otimes X_j + \sum A_j^* \otimes X_j^*.
\]
The  \df{free spectrahedron} determined by
$A$ is the sequence of sets $\cD_A= (\cD_A(n))$, where  \index{$\cD_A$}
\[
\cD_A(n) =\{X\in M_n(\C)^g: L_A(X)\succeq 0\}.
\]
Free spectrahedra arise naturally in applications such as systems engineering \cite{dOHMP09} and in the theories of matrix convex
sets, operator algebras, systems and spaces and completely positive maps \cite{EW,HKMjems,Pau}. They also provide tractable useful 
relaxations for spectrahedral inclusion problems that arise in semidefinite programming and engineering applications
such as the matrix cube problem \cite{BN02,HKMS}.

Given a tuple $E\in M_d(\C)^g$, the set \index{$\cB_E$}
\[
 \cB_E =\{X: \|\Lambda_E(X)\|\le 1\}
\]
 is a \df{spectraball} \cite{EHKM17,BMV}.  Spectraballs
 are special cases of free spectrahedra. Indeed,  it is readily
 seen that
\begin{equation*}
\cB_E=\cD_{\left(\begin{smallmatrix} 0&E\\0&0\end{smallmatrix}\right)}.
\end{equation*}

Let $M(\mathbb C)^g$ denote the sequence $(M_n(\mathbb C)^g)_n$.  A \df{subset} $\Gamma$ of $M(\mathbb C)^g$ is a sequence $(\Gamma_n)_n$ where $\Gamma_n \subset M_n(\mathbb C)^g$. (Sometimes we will write $\Gamma(n)$ in place of $\Gamma_n.$)  The subset $\Gamma$ is a \df{free set} if it is closed under direct sums and unitary similarity; that is,  if $X\in \Gamma_n$ and $Y\in \Gamma_m,$ then 
\[
 X\oplus Y = \begin{pmatrix} \begin{pmatrix} X_1& 0\\0 & Y_1 \end{pmatrix}, \dots, \begin{pmatrix} X_g& 0\\0 & Y_g \end{pmatrix} \end{pmatrix} \in \Gamma_{n+m};
\]
 and if $U$ is an $n\times n$ unitary matrix, then
\[
 U^* X  U = \begin{pmatrix} U^* X_1 U, \dots, U^* X_g U \end{pmatrix} \in \Gamma_n.
\]
 We say the free set $\Gamma=(\Gamma_n)_n$ is \df{open} if each $\Gamma_n$ is open. (Generally adjectives are applied levelwise to free sets unless noted otherwise.)  

A \df{free function} $f:\Gamma\to M(\C)$ is a sequence of functions
 $f_n:\Gamma_n\to M_n(\C)$ that \df{respects intertwining}; that is,
 if $X\in\Gamma_n$, $Y\in\Gamma_m$, $\SigmaT:\mathbb C^m\to\mathbb C^n$,
  and
 \[
  X \SigmaT=(X_1\SigmaT,\dots, X_g\SigmaT)
   =(\SigmaT Y_1,\dots, \SigmaT Y_g)=\SigmaT Y,
 \]
  then $f_n(X) \SigmaT =  \SigmaT f_m(Y)$. 
Assuming $\Gamma$ is an open free set, a free function $f:\Gamma\to M(\C)$ is
\df{analytic} if each $f_n$ is analytic. 
 Given free sets $\Gamma\subset M(\C)^g$ and $\Delta\subset M(\C)^h$, 
 a \df{free mapping} $f:\Gamma \to \Delta$ consists of free maps 
  $f^i:\Gamma \to M(\C)$ such that 
  $f(X) = \begin{pmatrix} f^1(X) & \dots & f^h(X)\end{pmatrix}$. In this 
  case we write $f=\begin{pmatrix} f^1 & \dots & f^h \end{pmatrix}$.
We refer the reader to   \cite{Voi04,KVV14}
for a fuller discussion of free sets and functions.

In this note,
we characterize the free
bianalytic maps $p:\cB_E\to \cD_B$  under some mild conditions on 
$E\in M_d(\C)^g$ and $B\in M_e(\C)^g$ and on  $p$ and its inverse $q$.
These free functions take a highly algebraic form that we call
\df{convexotonic}.
A tuple $\Xi=(\Xi_1,\dots,\Xi_g)\in M_g(\C)^g$  satisfying
\begin{equation*}
 \Xi_k \Xi_j = \sum_{s=1}^g (\Xi_j)_{k,s} \Xi_s
\end{equation*}
for each $1\le j,k\le g$ is \df{convexotonic}. 
 Convexotonic tuples naturally arise from finite
 dimensional algebras. If $\{J_1,\dots,J_g\}\subset M_d(\C)$
 is linearly independent and  spans an algebra, then 
 there exists a uniquely determined tuple $\Psi \in M_g(\C)^g$ such that
\begin{equation}
\label{e:JgetsXi}
  J_k J_j =\sum_{s=1}^g (\Psi_j)_{k,s} J_s
\end{equation}
 and Proposition \ref{prop:contable} says $\Psi$ is convexotonic.

Given a convexotonic tuple $\Xi\in M_g(\C)^g$, 
 the  expressions  $p=\begin{pmatrix} p^1 & \cdots & p^g\end{pmatrix}$ 
  and $q=\begin{pmatrix} q^1 & \cdots & q^g\end{pmatrix}$  whose components have the form
\begin{equation}
\label{e:tropici}
 p^i (x)=\sum_j x_j \left(I-\Lambda_{\Xi}(x) \right )^{-1}_{j, i} 
    \qquad \text{and} \qquad  q^i (x)=\sum_j x_j \left(I +\Lambda_{\Xi}(x) \right)^{-1}_{j,i},
\end{equation}
that is, in row form,
\begin{equation*}
 p(x)= x(I-\Lambda_\Xi(x))^{-1}
\qquad \text{and} \qquad q=x(I+\Lambda_\Xi(x))^{-1}
\end{equation*}
are, by definition, \df{convexotonic}.  The components of $p$ (resp. $q$)
are free functions with (free) domains consisting of those $X$ for which $I-\Lambda_\Xi(X)$ 
(resp. $I+\Lambda_\Xi(X)$) is invertible.  Hence $p$ and $q$ are
 free functions. 
 It turns out (see \cite[Proposition 6.2]{AHKM}) the mappings $p$ and $q$ are inverses 
of one another. %

Before continuing, we would like to point out that the component functions
$p^i$ of the convexotonic map $p$ of equation \eqref{e:tropici} 
are in fact free rational functions regular at $0.$ Accordingly we refer to
$p$ and $q$ as \df{birational} or free birational maps.  Free
rational functions are most easily described and naturally understood
 in terms  of realization theory as 
developed in the series of papers \cite{BGM05,BGM06a,BGM06b} of  Ball-Groenewald-Malakorn.
Indeed, based on those articles and on the results of \cite[Theorem 3.1]{KVV09}
and \cite[Theorem 3.5]{Vol17}) a \df{free rational function regular at $0$} 
can, for the purposes of this article,
 be defined with minimal overhead as an expression of the form
\begin{equation*}
r(x)= c^* \big(I-\Lambda_S(x)\big)^{-1} b
\end{equation*}
where $s$ is a positive integer,  $S\in M_s(\C)^g$ and  $b,c\in\C^s$ 
are vectors. The expression $r$ 
 is known as a realization. Realizations are easy
to manipulate and the theory of realizations is a powerful tool. 
The realization $r$ is  evaluated in the obvious fashion for
a tuple $X\in M_n(\C)^g$ as long as $I-\Lambda_S(X)$ is invertible. 
Free polynomials are free rational functions that are regular at $0$
and free rational functions regular at $0$ are stable with respect to the formal
algebraic operations of addition, multiplication and inversion  in the sense that
if  $r$ is a free rational function regular at $0$ and $r(0)\ne 0$, 
 then its  multiplicative inverse $r^{-1}$
 is also a free rational function regular at $0$. 
Thus, expressing $p^i$ as
\[
 p^i = \sum_{s=1}^g x_s  e_s^*(I-\Lambda_\Xi(x))^{-1} e_i
\]
shows it is a free rational function regular at $0$.

To state our main theorem precisely we need a bit more terminology.
A subset $\{u^1,\dots,u^{d+1}\}$ of $\C^d$ is a \df{hyperbasis} for $\C^d$ if each
$d$ element subset is a basis.  The tuple $A\in
M_d(\C)^g$ is \df{sv-generic} if there exists
$\alpha^1,\dots,\alpha^{d+1}$ and $\beta^1,\dots,\beta^d$ in $\C^g$
such that, for each $1\le j\le d+1$,
 the matrix $I-\Lambda_A(\alpha^j)^*\Lambda_A(\alpha^j)$ is positive
semidefinite, has a one-dimensional kernel spanned by $u^j$ and the
set $\{u^1,\dots,u^{d+1}\}$ is a hyperbasis for $\C^d$; and, for each
$1\le k\le g$, the matrix
$I-\Lambda_A(\beta^k)\Lambda_A(\beta^k)^*$ is positive semidefinite,
has a  one-dimensional kernel  spanned by $v^k$ and the set $\{v^1,\dots,v^d\}$ is a basis
for $\C^d$.  
Generic  tuples  $A$ satisfy this property, see  
\cite[Remark 7.5]{AHKM}.
 Given a matrix-valued free analytic polynomial $Q$, the set
\[
 \cG_Q = \{X\in M(\C)^g: \|Q(X)\|<1\} \subset M(\C)^g
\]
 is a \df{free pseudoconvex} set.

\begin{theorem}
 \label{thm:main}
 Suppose $E\in M_d(\C)^g$ and $B\in M_e(\C)^g.$ If
\begin{enumerate}[\rm (i)]
 \item $E$ is sv-generic and linearly independent; 
 \item $B$ \label{i:main2}
   is sv-generic and $\cD_B$ is bounded;
 \item \label{i:main3}
      $p:\cB_E\to\cD_B$ is bianalytic with $p(0)=0$ and $p^\prime(0)=I;$ and 
 \item  \label{i:main4}
   $p$ is defined on a pseudoconvex domain containing $\cB_E$ and $q:\cD_B\to \cB_E$, the inverse of $p$, is defined on a pseudoconvex domain containing
 $\cD_B,$ 
\end{enumerate}
then there exist  $g\times g$ unitary matrices
 $Z$ and $M$ and a tuple  $\Xi \in M_g(\C)^g$ such that 
\begin{enumerate}[\rm (1)]
 \item \label{it:BMZEM} $B= M^* ZEM$;
 \item \label{it:secret}for each $1\le j,k\le g$, 
 \begin{equation}
   \label{it:secretalg}
  E_k Z  E_j = \sum_s (\Xi_j)_{k,s} E_s;
\end{equation}
 \item 
   \label{it:supersecret}
     the tuple $B$ spans an algebra and  
\[
 B_k B_j = \sum_s (\Xi_j)_{k,s} B_s;
\]
 \item $\Xi$ is convexotonic and  $p$ is the corresponding  convexotonic map
     $p=x(I-\Lambda_\Xi(x))^{-1}$.
\end{enumerate}
\end{theorem}

\begin{remark}\rm
\label{r:inorder}
Several remarks are in order.
\begin{enumerate}[(i)]
\item 
 A free spectrahedron $\cD$ is \df{sv-generic} if there exists an sv-generic
 tuple $A$ such that $\cD=\cD_A$.
The article \cite{AHKM} contains a version of Theorem \ref{thm:main} for bianalytic
 mappings between sv-generic  free spectrahedra 
(actually a weaker, but more complicated to formulate, condition from \cite{AHKM} 
that we call eig-generic would also suffice here). The sv-generic free
spectrahedra are  in fact generic among free spectrahedra
in the sense of algebraic geometry. 
 However, spectraballs, within the class of free spectrahedra, are {\it never}
 sv-generic in view of Lemma \ref{l:ballnotsv}. 
 Hence, Theorem \ref{thm:main} extends Theorem  \cite[Theorem 1.8]{AHKM},
 to the important special case of maps from spectraballs to free spectrahedra. 
\item 
  Let 
\begin{equation}
\label{eq:F}
 F_1=\begin{pmatrix} 0 &1&0\\0&0&1\\0&0&0\end{pmatrix} = \begin{pmatrix} 0 & E_1 \\ 0& 0\end{pmatrix}
 \qquad \text{and} \qquad F_2=F_1^2 =\begin{pmatrix} 0 & E_2 \\ 0 & 0\end{pmatrix},
\end{equation}
where the tuple $E$ is given in equation \eqref{eq:E}.
 The tuple $F$ is nilpotent. Thus, by Lemma \ref{l:ballnotsv}, 
 it is not sv-generic and 
   the results of \cite{AHKM} 
  do not apply to bianalytic maps $r:\cD_F\to\cD_B$  with
  $r(0)=0$ and $r^\prime(0)=I$.  Moreover,  $\cD_F$ 
 is not a spectraball by Proposition \ref{prop:model} and thus Theorem \ref{thm:main} does not
 directly apply either.
  However,  as we show, there is an sv-generic tuple $E$ 
  and a bianalytic map $p:\cB_E\to\cD_F$
  with $p(0)=0$ and $p^\prime(0)=I$  (see Proposition \ref{prop:model}). 
  On the other hand, Theorem \ref{thm:main} does
  apply to bianalytic maps $f:\cB_E\to \cD_B$. By composing $f$ with
   $p^{-1}$,  Propositions
  \ref{prop:BZZZZ} and \ref{lem:isto}  classify the choices for $B$ and
   all bianalytic maps 
  between $\cD_F$ and $\cD_B$. In particular, these maps are convexotonic.  
\item  It is easy to check that item \eqref{it:BMZEM} implies $\cB_E=\cB_B$. 
\item Since $E$ is assumed linearly independent $\Xi$ is uniquely determined by
  equation \eqref{it:secretalg}. Further,  by Proposition \ref{prop:contable},  $\Xi$
 is convexotonic.
\item   Items \eqref{it:secret} and \eqref{it:supersecret} are equivalent given \eqref{it:BMZEM}.
\item Note that, while $p$ is only assumed to be bianalytic, the conclusion is that 
 $p$ is birational, a phenomena encountered 
 frequently in rigidity theory in several complex variables, cf.~\cite{For}.
\item  A key ingredient in the proof of Theorem \ref{thm:main} is a suitable Positivstellensatz.
Namely, $p$ maps $\cB_E$ into $\cD_B$ 
if and only if
$L_B(p(X))\succeq0$ for all $X\in\cB_E$,
and this equivalence feeds naturally into Positivstellens\"atze, a pillar of real algebraic geometry.
The one used here (from \cite{AHKM}) is related to that of \cite{AM14}, which was  developed in full generality in \cite{BMV}.\looseness=-1
\item  An easy argument shows, for $A\in M_d(\C)^g,$ if $\cD_A$ is bounded, then $A$ (really $\{A_1,\dots,A_g\})$ is linearly independent \cite[Proposition 2.6(2)]{HKM}.
The converse fails in general; e.g., if each $A_j$ is positive semidefinite.  On the other hand,  $E\in M_d(\C)^g$
is linearly independent if and only if  $\cB_E$ is bounded \cite[Proposition 2.6(1)]{HKM}.
\end{enumerate}
\end{remark}

There is a natural converse to Theorem \ref{thm:main}.
Let $\interior(\cD_A)$ and $\interior(\cB_A)$ denote the interiors of $\cD_A$ and $\cB_A$ respectively.
Recall a mapping  between metric spaces is \df{proper} if the inverse image of compact sets are compact.
Thus, for open sets $\cU\subset M(\C)^g$ and $\cV\subset M(\C)^h$, a free mapping $f:\cU\to \cV$ is proper if each
$f_n :\cU_n\to \cV_n$ is proper.

\begin{prop}\label{prop:properobvious}
Suppose $\JJ\in M_d(\C)^g$ is linearly independent,  spans an algebra, $\Xi$ is the resulting
convexotonic tuple,
\begin{equation*}
 \JJ_k \JJ_j = \sum_{s=1}^g (\Xi_j)_{k,s} \JJ_s,
\end{equation*}
and $\qq$ is the convexotonic (birational) map,
\[
 \qq(x) = x(I+\Lambda_{\Xi}(x))^{-1}.
\]
Then
\begin{enumerate}[\rm (1)]
 \item The domain of $\qq$ contains $\cD_{\JJ}.$
 \item $\qq$ is a bianalytic map between $\interior(\cD_{\JJ})$ and 
  $\interior(\cB_{\JJ})$; that is $\pp$, the (convexotonic) inverse of 
   $\qq,$ maps $\interior(\cB_{\JJ})$ 
  into $\interior(\cD_{\JJ}).$ In particular, $\qq$ is proper. 
 \item $\qq$ maps the boundary of $\cD_{\JJ}$ into the boundary of $\cB_{\JJ}$;
 \item if, in addition, $\cD_{\JJ}$ is bounded, then $\qq$ is a bianalytic map
  between $\cD_{\JJ}$    and $\cB_{\JJ}.$ In particular, the domain of $p$ 
  contains $\cB_{\JJ}$.
\end{enumerate}
\end{prop}

In case $\JJ$ does not span an algebra, we have the following corollary
of Proposition \ref{prop:properobvious}.

\begin{cor}\label{cor:obvious}
Let $A\in M_d(\C)^g$ and assume $A$ is  linearly independent 
(e.g.~$\cD_A$ is bounded). 
Let $\cA$ denote the algebra spanned by the tuple $A$. If
 $C_{1},\ldots,C_h\in M_d(\C)$
and the tuple $\JJ=(\JJ_1, \dots, \JJ_{g+h})=(A_1,\dots,A_g,C_{1},\dots,C_h)$ 
is linearly independent and spans $\cA$, then
 there is a rational map $\GG$ with 
$\GG(0)=0$ and $\GG^\prime(0)=I$ such that
\begin{enumerate}[\rm (1)]
 \item  $\GG$ is an injective proper map from  $\interior(\cD_A)$  
  into $\interior(\cB_{\JJ});$ and 
 \item $\GG$ maps the  boundary of $\cD_A$ into boundary of $\cB_{\JJ}.$
\end{enumerate}
 Further, the tuple 
$\Xi\in M_{g+h}(\C)^{g+h},$  uniquely determined by 
\begin{equation*}
 \JJ_k \JJ_j = \sum_{s=1}^h (\Xi_j)_{k,s} \JJ_s,
\end{equation*}
 is convexotonic and 
\begin{equation*}
 \GG(x) = \begin{pmatrix} x_1 & \cdots & x_g & 0 & \cdots &0 \end{pmatrix} 
    \, \left ( I+\sum_{j=1}^g \Xi_j x_j \right )^{-1}.
\end{equation*}
\end{cor}

 For further results, not already cited, on 
 free bianalytic and proper free analytic maps 
 see \cite{HKMS09, HKM11a, HKM11b,Po2,KS,MS} and the references therein. 

 The remainder of the article is organized as follows.  
 Proposition \ref{prop:properobvious} and Corollary \ref{cor:obvious} are
 established in Section \ref{sec:obvious}. Theorem \ref{thm:main} is 
 proved in Section \ref{sec:proof}.
 The article concludes with several examples; see Section \ref{sec:prebasic}.

\section{Proof of Proposition \ref{prop:properobvious}}
\label{sec:obvious}
 This section gives the proof of Proposition \ref{prop:properobvious}.
 Implicit in the statement of that result, and used in the proof of
 Theorem \ref{thm:main}, is the connection between finite dimensional algebras
 and convexotonic tuples described in the following proposition.

\begin{prop}
\label{prop:contable}
   Suppose $\FG\in M_{d\times e}(\C)^g$  and $\{\FG_1,\dots,\FG_g\}$
 is linearly independent,
 $C\in M_{e\times d}(\C)$ and  $\Psi\in M_g(\C)^g$. If
\begin{equation}
\label{e:FGgetsXi}
  \FG_\ell C \FG_j =\sum_{s=1}^g (\Psi_j)_{\ell,s} \FG_s,
\end{equation}
 then the tuple $\Psi$ is convexotonic.  In particular,
 if $J\in M_d(\C)^g$ is linearly independent and spans an algebra,
 then the tuple $\Psi$ uniquely determined by equation 
 \eqref{e:JgetsXi} is convexotonic.
\end{prop}

\begin{proof}
 For notational ease let $T=C\FG \in M_e(\C)^g$.
  The hypothesis implies $T$ spans an algebra
 (but not that $T$ is linearly independent).  Routine calculations give
\[
 [\FG_\ell T_j] T_k = \sum_{t=1}^g (\Psi_j)_{\ell,t} \FG_t\, T_k 
 = \sum_{s,t=1} (\Psi_j)_{\ell,t} (\Psi_k)_{t,s} \FG_s
  = \sum_s (\Psi_j\, \Psi_k)_{\ell,s} \FG_s.
\]
 On the other hand
\[
 \FG_\ell [T_j T_k] = \FG_\ell C [\FG_j T_k] =  \sum_t \FG_\ell (\Psi_k)_{j,t}T_t 
  = \sum_{s,t}  (\Psi_t)_{\ell,s} (\Psi_k)_{j,t} \FG_s.
\]
 By independence of $\FG$, 
\[
(\Psi_j \Psi_k)_{\ell,s} =  \sum_{t}  (\Psi_k)_{j,t}  (\Psi_t)_{\ell,s}
\]
 and therefore
\[
\Psi_j \Psi_k  = \sum_{t}  (\Psi_k)_{j,t}  \Psi_t
\] 
 and the proof is complete.
\end{proof}

\begin{lemma}
 \label{lem:inverseok}
   Suppose $F\in M_d(\C)^g$.
    If $I+\Lambda_F(X)+\Lambda_F(X)^* \succeq 0$, then $I+\Lambda_F(X)$
    is invertible. 
\end{lemma}

\begin{proof}
 Arguing the contrapositive, suppose $I+\Lambda_F(X)$ is not invertible. 
 In this case
 there is a unit vector $\gamma$ such that
\[
 \Lambda_F(X)\gamma = -\gamma.
\]
 Hence,
\[
 \langle (I+\Lambda_F(X)+\Lambda_F(X)^*)\gamma,\gamma\rangle
  = \langle \Lambda_F(X)^*\gamma,\gamma\rangle = 
 \langle \gamma,\Lambda_F(X)\gamma\rangle = 
  -1. \qedhere
\]
\end{proof}

\begin{lem}\label{lem:onehalf}
Let $T\in M_d(\C)$. Then
\begin{enumerate}[\rm (a)]
\item $I+T+T^*\succeq0$ if and only if $I+T$ is invertible and
$\|(I+T)^{-1}T\|\leq1$;
\item $I+T+T^*\succ0$ if and only if
$I+T$ is invertible and $\|(I+T)^{-1}T\|<1$.
\end{enumerate}

Similarly if $I-T$ is invertible, then  $\|T\|\le 1$ if and only if 
$I+R+R^*\succ 0$, where $R=T(I-T)^{-1}$.
\end{lem}

\begin{proof}
(a) We have the following chain of equivalences:
\[
\begin{split}
\|(I+T)^{-1}T\|\leq1 \quad & \iff
\quad 
I- \big((I+T)^{-1}T\big)\big((I+T)^{-1}T\big)^*\succeq0
\\
& \iff \quad 
I- (I+T)^{-1}T T^*(I+T)^{-*}\succeq0  \\
 & \iff \quad
(I+T)(I+T)^* - TT^*\succeq0 \\
& \iff \quad
I+T+T^*\succeq0.
\end{split}
\]

The proof of (b) is the same. 
\end{proof}

\begin{prop}\label{prop:inverseok}
   For $F\in M_d(\C)^g$ we have
   \[
\cD_F=\{X\colon \|(1+\La_F(X))^{-1}\La_F(X)\|\leq1\}.
   \]
\end{prop}

\begin{proof}
Immediate from Lemma \ref{lem:onehalf}.
\end{proof}

\begin{proof}[Proof of Proposition~\ref{prop:properobvious}]
Let $q$ denote the convexotonic map associated to the convexotonic
tuple $\Xi$ in the statement of the proposition,
\[
\rr(x)= \begin{pmatrix} x_1 &\cdots & x_{g}\end{pmatrix} 
  \left ( I+\Lambda_\Xi(x) \right )^{-1} 
    =x \left ( I+\Lambda_\Xi(x) \right )^{-1}.
\]
Compute
\[
\begin{split}
 \Lambda_{\JJ}(\rr(x))\, \Lambda_\JJ(x)
  = &  \sum_{s,k=1}^g q^s(x) x_k \JJ_s \JJ_k 
 = \sum_{j=1}^g  \sum_{s=1}^g q^s(x) \left [ \sum_{k=1}^g x_k (\Xi_k)_{s,j} \right ] \JJ_j \\
 =& \sum_{j=1}^g \sum_{s=1}^g q^s(x) (\Lambda_{\Xi}(x))_{s,j} \JJ_j 
=  \sum_{j=1}^g \sum_{t=1}^g x_t  \left [ \sum_{s=1}^g (I+\Lambda_\Xi(x))^{-1}_{t,s} (\Lambda_{\Xi}(x))_{s,j} \right ] \JJ_j \\ 
 = & \sum_{j=1}^g   \sum_{t=1}^g x_t [(I+\Lambda_\Xi(x))^{-1} \Lambda_\Xi(x)]_{t,j} \JJ_j.
\end{split}
\]
Hence,
\[
\Lambda_\JJ(\rr(x))\, (I+\Lambda_\JJ(x))
=  \sum_{j=1}^g   \sum_{t=1}^g x_t [(I+\Lambda_\Xi(x))^{-1} (I+\Lambda_\Xi(x))]_{t,j} \JJ_j 
=  \Lambda_{\JJ}(x).
\]
Thus, as free (matrix-valued) rational functions regular at $0$,
\begin{equation}
 \label{eq:LJr}
 \Lambda_{\JJ}(\rr(x)) = (I+\Lambda_{\JJ}(x))^{-1} \, \Lambda_{\JJ}(x)=:F(x).
\end{equation}

Since $J$ is linearly independent, given $1\le k\le g$, 
there is a linear functional $\lambda$ such that 
 $\lambda(J_j)=0$ for  $j\ne k$ and $\lambda(J_k)=1$. Applying
$\lambda$ to equation \eqref{eq:LJr}, gives
\begin{equation}
\label{eq:lambdaLJr}
 q^k(x) = \lambda(F(x)).
\end{equation}
Since $\lambda(F(x))$ is a free rational function whose domain contains
\[
\mathscr{D} = \{X: I+\Lambda_{\JJ}(X) \mbox{ is invertible}\},
\]
the same is true for $q^k$. (As a technical matter, each side of equation
\eqref{eq:lambdaLJr} is a rational expression. Since they are defined and
 agree on a neighborhood of $0$, they determine the same free rational
function. It is the domain of this rational function that contains
 $\mathscr{D}.$ See \cite{Vol17}, and also \cite{KVV09}, for full details.)
By Lemma \ref{lem:inverseok}, $\mathscr{D}$ contains $\cD_\JJ$ 
 (as $X\in \cD_\JJ$ implies $I+\Lambda_{\JJ}(X)$ is invertible). Hence
 the domain of the free rational mapping  $\rr$ contains   $\cD_\JJ$.
By Lemma \ref{lem:onehalf} and equation \eqref{eq:LJr}, $\rr$ maps the interior of $\cD_{\JJ}$ into the 
interior of $\cB_{\JJ}$ and the boundary of $\cD_{\JJ}$ into
the boundary of $\cB_{\JJ}$.

Similarly, 
\begin{equation}
\label{eq:LJp}
 (I-\Lambda_{\JJ}(x))^{-1} \, \Lambda_\JJ(x) =\Lambda_{\JJ}(\pp(x)),
\end{equation}
where $\pp(x) = x(I-\Lambda_{\Xi}(x))^{-1}$. Arguing as above shows
 the domain of $\pp$ contains the set
\[
\mathscr{E} =\{X: I-\Lambda_{\JJ}(X) \mbox{ is invertible}\},
\]
which in turn contains $\interior(\cB_{\JJ})$
(since $\|\Lambda_{\JJ}(X)\|<1$ allows for an
application of Lemma \ref{lem:onehalf}). By Lemma \ref{lem:onehalf}
and equation \eqref{eq:LJp},
 $\pp$ maps the interior of
$\cB_{\JJ}$ into the interior of $\cD_{\JJ}.$ 
Hence $\rr$ is bianalytic between
these interiors.  Further, if  $X$ is in the boundary
of $\cB_{\JJ}$, then for $t\in \mathbb C$ and $|t|<1$, 
we have $\pp(tX)\in \interior(\cD_{\JJ})$ and
\[
 \Lambda_J(\pp(tX))= (I-\Lambda_{\JJ}(tX))^{-1} \, \Lambda_\JJ(tX).
\]
Assuming $\cD_{\JJ}$ is bounded, it follows that
  $I-\Lambda_{\JJ}(X)$ is invertible and thus $X$ is in the domain
of $\pp$ and $\pp(X)$ is in the boundary of $\cD_{\JJ}$.
\end{proof}

\begin{proof}[Proof of Corollary~\ref{cor:obvious}]
Letting $z=(z_1,\dots,z_{g+h})$ denote a $g+h$ tuple of freely non-commuting 
indeterminants,  and $\Xi$ the convexotonic $g+h$ tuple as described in the 
corollary, by Proposition \ref{prop:properobvious} the birational mapping
\[
 q(z) = z (I+\Lambda_{\Xi}(z))^{-1}
\]
 is a  bianalytic (hence injective and proper) mapping between 
 $\interior(\cD_{\JJ})$ and 
 $\interior(\cB_{\JJ})$ that also maps boundary to boundary.
 The mapping $\iota:\cD_A\to\cD_{\JJ}$ defined by 
 $\iota(x)=(x,0)$  is proper 
 from $\interior(\cD_A)$ to $\interior(\cD_{\JJ})$ and maps
 boundary to boundary. Hence, the composition 
\[
r(x) = \pp(\iota(x))= \begin{pmatrix} x & 0 \end{pmatrix}
 \,  (I-\Lambda_{\Xi}(x,0))^{-1}
\]
 is a proper map from $\interior(\cD_A)$ into $\interior(\cB_{\JJ})$ that also
maps boundary to boundary.
\end{proof}

\section{Proof of Theorem \ref{thm:main}}
\label{sec:proof}
 Given $E\in M_d(\C)^g$,  let
\begin{equation*}
 A=\begin{pmatrix} 0 & E\\ 0& 0 \end{pmatrix} \in M_{2d}(\C)^g.
\end{equation*}
Thus $\cB_E=\cD_A$ and, 
 among other things, by assumption, there is a bianalytic map $p:\cD_A\to \cD_B$.
It follows by the analytic Positivstellensatz \cite[Theorem 1.9]{AHKM}
 applied to the matrix-valued free analytic function
\[
 G(x) =\Lambda_B(p(x))
\]
that there exists
 a Hilbert space $H$, an isometry $\wtC$ on the range of $\IHA$
 and an isometry $\sW:\mathbb C^e \to  H\otimes \mathbb C^{2d}$  such that,
 with $\tR=(\wtC-I)(\IHA)$, 
\begin{equation}
 \label{eq:possS}
   L_{B}(p(x)) = I+G(x)^* + G(x)=
      \sW^*(I-\Lambda_{\tR}(x))^{-*} L_{\IHA}(x) (I-\Lambda_{\tR}(x))^{-1}\sW.
\end{equation}
 That the analytic Positivstellens\"atze requires  $\cD_A$ to be bounded and 
 $G$ to  extend analytically to a pseudoconvex set
 containing $\cD_A$ explains the need for the hypotheses that
 $\cD_A=\cB_E$ is bounded (equivalently $E$ is linearly independent)
 and $p$ extends analytically to a pseudoconvex set containing $\cD_A$. 

Since $E$ is sv-generic, both $\ker(E):=\cap \ker(E_j)=\{0\}$ and $\ker(E^*)=\{0\}.$
In particular, \index{$\rg$}
\[
 \rg(A) := \spann(\bigcup_{j=1}^g  \rg(A_j)) 
  = \C^d \, \oplus \, \{0\}.
\]
In particular, $\dim(\rg(A))=d$. Likewise $\dim(\rg(A^*))=d$ too. 

The next step involves a call to  \cite[Lemma 7.7]{AHKM}. 
That lemma is stated in terms of conditions referred to as
eig-generic, weakly eig-generic, $*$-generic and weakly $*$-generic
formally defined in \cite[Definition 7.3]{AHKM}. It is readily seen
that if a  $g$-tuple $F$ of $N\times N$ matrices 
 is sv-generic, then it is both eig-generic and
$*$-generic (and thus  weakly eig-generic and weakly $*$-generic).
In particular,  $\rg(F)=\C^N=\rg(F^*)$.
Thus, both $E$ and $B$ are both eig-generic and $*$-generic. 
That $E$ is eig-generic implies $A$ is weakly eig-generic;
 and that $E$ is $*$-generic implies $A$ is weakly $*$-generic.

 By \cite[Lemma 7.7(1)]{AHKM}, $d=\dim(\rg(A^*))\le \dim(\rg(B^*))=e.$
 Applying \cite[Lemma 7.7(1)]{AHKM} to  $q:\cD_B\to\cD_A$ (so reversing the roles 
  of $A$ and $B$), it also follows that $e\le d$. Hence $\dim(\rg(A^*))=d=e=\dim(\rg(B^*))$
 and  $\dim(\rg(A))=d=e=\dim(\rg(B))$. 
Thus we may now invoke (the weakly version of) \cite[Lemma 7.7(4)]{AHKM} that says
   there is a vector $\lambda\in H$ and a unitary
  $M:  \rg(B^*)   \to \rg(A^*)$
  and an
   isometry $N:     \rg(B^*)   \cap \rg(B)
   \to \rg(A)$
   such that  $\sW v=\lambda\otimes \iota Mv$ for $v\in \rg(B^*)$
  and $\wtC(\lambda \otimes \iota Nv) = \lambda\otimes \iota Mv$ for
  $v\in \rg(B^*)\cap \rg(B)$ (where we over use $\iota$, letting it denote the inclusions
 $\rg(A^*)\subset \C^{2d}$ and $\rg(B^*)\subset \C^{2d}$).
This general statement  in our case specializes,
because   $\rg(B^*) = \C^d  $, $  \rg(B) =\C^d$ and
$\dim(\rg(A))=d$, to give 
\begin{enumerate}[(i)]
 \item  $M:  \C^d   \to \rg(A^*)$ is unitary;
 \item  $N:     \C^d   \to \rg(A)$ is unitary;
 \item  $ \sW v=\lambda\otimes \iota Mv$ for $v\in \C^d$; and 
 \item  $\wtC(\lambda \otimes \iota Nv) = \lambda\otimes \iota Mv$ for $v\in \C^d.$
\end{enumerate}

It follows that there is a unitary mapping $Z:\rg(A)\to \rg(A^*)$ such that,
for $w\in \rg(A)$,
\[
  \wtC(\lambda\otimes w) = \lambda\otimes \iota Z w.
\]
  Let $[\lambda]=\C \lambda$, the one-dimensional subspace of $H$ spanned
by the unit
  vector $\lambda$.

 Let 
\begin{equation*}
  C=\begin{pmatrix} 0&0\\ Z&0 \end{pmatrix}.
\end{equation*}
  In particular $C$ is isometric on the range of $A$.
Let $R=(C-I)A$. 
   For $1\le j\le d$, and $\gamma \in \C^{2d}$,
\begin{equation*}
 \begin{split}
 \tR_j (\lambda\otimes \gamma) & =  (\wtC-I)[\IHA_j] (\lambda \otimes \gamma) \\
 & =  (\wtC-I) (\lambda \otimes A_j\gamma) \\
 & =  \lambda \otimes (\iota Z-I)A_j \gamma\\
 & =  \lambda \otimes (C-I)A_j \gamma \\
 & =  \lambda\otimes R_j\gamma.
\end{split}
\end{equation*}
 Thus $[\lambda]\otimes \mathbb C^{2d}$ is invariant for the tuple $\tR$ and further
\[
 \tR_j (\lambda \otimes I) = \lambda\otimes R_j.
\]
 It follows that $[\lambda]\otimes \C^{2d}$ is invariant for the mapping
  $(I-\Lambda_{\tR}(x))^{-1}$ and moreover,
\[
 (I-\Lambda_{\tR}(x))^{-1} (\lambda\otimes I) = \lambda \otimes (I-\Lambda_R(x))^{-1}.
\]
 Finally, since $\sW$ maps into $[\lambda]\otimes \C^{2d}$ and $\sW\gamma =\lambda\otimes  \iota M\gamma$, 
\[
 W(x):=(I-\Lambda_{\tR}(x))^{-1}\sW  = \lambda\otimes (I-\Lambda_R(x))^{-1} \iota M.
\]
 
 Since also $[\lambda]\otimes \C^{2d}$ is invariant for $L_{\IHA}(x)$, 
\[
 L_{\IHA}(x) W(x) = \lambda \otimes L_A(x) (I-\Lambda_R(x))^{-1} \iota M.
\]
 Returning to equation \eqref{eq:possS} and using $\lambda^*\lambda =1$,
\begin{equation}
 \label{eq:poss one term}
\begin{split}
 L_{{B}}(p(x)) & =  (\lambda^* \otimes (\iota M)^* (I-\Lambda_R(x))^{-*} (\lambda\otimes L_A(x) (I-\Lambda_R(x))^{-1} \iota M \\
 & = M^* \iota^* (I-\Lambda_R(x))^{-*} L_A(x) (I-\Lambda_R(x))^{-1} \iota M.
\end{split}
\end{equation}
 Comparing the coefficients of the $x_j$ terms in equation \eqref{eq:poss one term} gives
\[
 B = M^* \iota^* C A \iota M.
\]
 Since $M:\C^d\to \rg(A^*)$ is unitary, 
\begin{equation}
\label{eq:iM}
 \iota M= \begin{pmatrix} 0 \\ U \end{pmatrix} : \C^d\to \C^{2d}=\C^d\oplus \C^d = \rg(A)\oplus \rg(A^*)
\end{equation}
 for a unitary mapping $U:\C^d\to\C^d$.  Thus, 
\[
 B = U^* Z E U.
\]
 
 Since
\[
  R= (C-I)A = \begin{pmatrix} 0& -E \\ 0 & ZE \end{pmatrix}.
\]
 it follows that
\[
 F(x):=(I-\Lambda_R(x))^{-1} = \begin{pmatrix} I & - \Lambda_E(x) (I-\Lambda_{ZE}(x))^{-1} \\ 0& (I-\Lambda_{ZE}(x))^{-1} \end{pmatrix}.
\]
 Consequently, 
{\footnotesize
\begin{multline*}
 F(x)^* L_A(x) F(x) \\
  =  \begin{pmatrix} I&0\\ -(I-\Lambda_{ZE}(x))^{-*} \Lambda_E(x)^* & (I-\Lambda_{ZE}(x))^{-*}\end{pmatrix} \,
  \begin{pmatrix} I & \Lambda_E(x) \\ \Lambda_E(x)^* & I \end{pmatrix}\, 
  \begin{pmatrix} I & - \Lambda_E(x)(I-\Lambda_{ZE}(x))^{-1} \\ 0& (I-\Lambda_{ZE}(x))^{-1} \end{pmatrix} \\
 =  \begin{pmatrix} I &0 \\ 0 & (I-\Lambda_{ZE}(x))^{-*} (I-\Lambda_E(x)^* \Lambda_E(x))(I-\Lambda_{ZE}(x))^{-1} 
    \end{pmatrix}.
\end{multline*}
}
 Hence, from equations \eqref{eq:poss one term} and \eqref{eq:iM},
\[
  L_B(p(x))= U^* (I-\Lambda_{ZE}(x))^{-*} (I-\Lambda_E(x)^* \Lambda_E(x))(I-\Lambda_{ZE}(x))^{-1} U.
\]
  Further, letting
 \begin{equation*}
  \widetilde{B} = CA =\begin{pmatrix} 0 & 0 \\ 0 & ZE \end{pmatrix},
 \end{equation*}
we have
\[
 L_{\widetilde{B}}(p(x)) = \cW^* F(x)^* L_A(x) F(x) \cW
\]
 where 
\[
 \cW =\begin{pmatrix} I & 0\\0 & U\end{pmatrix}.
\]
  By \cite[Theorem 6.7]{AHKM}, $p$ is a convexotonic mapping determined by the (uniquely determined) convexotonic tuple 
  $\Xi$ satisfying
\[
 A_k (C-I)A_j = \sum_{s=1}^g (\Xi_j)_{k,s} A_s.
\]
 Equivalently, 
\begin{equation}
 \label{eq:secretalg}
 E_kZE_j = \sum_{s=1}^g (\Xi_j)_{k,s} E_s.
\end{equation}

Finally to prove item \eqref{it:supersecret}, multiply equation \eqref{eq:secretalg} by $Z$ on the left and
 use $B=U^*ZEU$ to obtain,
\[
 B_k B_j = \sum_{s=1}^g (\Xi_j)_{k,s} B_s.
\]

\section{Examples}\label{sec:prebasic}
  In this section we take up some examples that motivate Theorem \ref{thm:main} and 
Corollary \ref{cor:obvious}.  First we show that a spectraball, as a member of the 
class of free spectrahedra, is never sv-generic.

\begin{lemma}
\label{l:ballnotsv}
  Suppose $B\in M_d(\C)^g$.  
\begin{enumerate}[(a)]
\item \label{i:notsv1}
     If  $B$ is sv-generic, then $\ker(B):=\cap_{j=1}^g \ker(B_j)= \{0\}.$
\item \label{i:notsv2}
    If $B$ is nilpotent, then $\ker(B)\ne \{0\}$.
\item \label{i:notsv3}
  If $B$ is nilpotent, then $\cD=\cD_B$ is not sv-generic.
\item \label{i:notsv4}
    If $\cD$ is a spectraball, then $\cD$ is not sv-generic.
\end{enumerate}
\end{lemma} 

\begin{remark}\rm
 \label{r:ballnotsv}
 In fact item \eqref{i:notsv1}, and thus items \eqref{i:notsv3} and
 \eqref{i:notsv4},  remain true with sv-generic replaced by eig-generic
 \cite[Definition 7.3]{AHKM}.
\end{remark}

\begin{proof}
 If $\alpha \in \C^g$, $u\in \C^d$ and $[I-\Lambda_B(\alpha)^* \Lambda_B(\alpha)]u=0$, then
 $u\in \rg(B^*)=\ker(B)^\bot$. Hence, if $B$ is sv-generic, then there exists a basis
 $\{u^1,\dots,u^d\}$ of $\C^d$ such that each $u^j \in \rg(B^*)$. Thus 
 $\C^d = \rg(B^*) = \ker(B)^\bot$  and therefore $\ker(B)=\{0\}$.

 Now suppose $B$ is nilpotent.  Thus there is an $N$ 
 such that if $\beta$ is a word whose length exceeds $N$, then
  $B^\beta=0$. Hence there is a word $\alpha$ (potentially empty)
 such that $B^\alpha \ne 0$, but $B_j B^\alpha =0$ for $1\le j\le g$. 
 It follows that $\{0\}\ne \rg(B^\alpha) \subset \ker(B)$, proving
 item \eqref{i:notsv2}.

 To prove item \eqref{i:notsv3}, suppose $B$ is nilpotent and  let
 $\cD=\cD_B.$   Let $M\in M_m(\C)^g$ be a \df{minimal defining tuple} for
 $\cD$, meaning $\cD=\cD_M$ and if $C\in M_s(\C)^g$ and $\cD=\cD_C$, then $s\ge m$.
 By \cite[Proposition 2.2]{EHKM17}, there is a tuple $J$ such that $B$ is unitarily
 equivalent to $M\oplus J$. Since $B$ is nilpotent, so is $M$.  Hence
 $\ker(M)\ne \{0\}$ by item \eqref{i:notsv2}.  Now
 suppose $C$ is any other tuple so that $\cD_C=\cD_B$. 
 Another application of \cite[Proposition 2.2]{EHKM17} gives
 a tuple $N$ such that $C$ is unitarily equivalent 
 to  $M\oplus N$. Hence $\ker(C)\ne \{0\}$ and by item \eqref{i:notsv1},
 $C$ is not sv-generic. Thus $\cD=\cD_B$ is not sv-generic.

 Finally suppose $\cD$ is a spectraball.  Hence there is a
 positive integer $e$ and tuple $E\in M_e(\C)^g$ 
 such that $\cD=\cB_E$.  Since $\cD=\cB_E = \cD_A$, where 
\[
 A =\begin{pmatrix} 0&E\\ 0 & 0 \end{pmatrix} \in M_{2e},
\]
 and $A$ is nilpotent, item \eqref{i:notsv3} implies $\cD$ is not sv-generic.
\end{proof}

\subsection{A spectrahedron defined by a nilpotent tuple}
\label{sec:basic}
A spectrahedron defined by a nilpotent tuple cannot have 
sv-generic coefficients according to Lemma \ref{l:ballnotsv}, 
but we give an example here of how one can overcome this problem
by mapping to a spectraball.

 Let
\begin{equation}
\label{eq:E}
  E_1=I_2  \qquad \text{and} \qquad
  E_2=\begin{pmatrix} 0&1\\0&0\end{pmatrix} 
\end{equation}
and let $F$ denote the $2$-tuple of $3\times 3$ matrices
given in equation \eqref{eq:F}.
Note that $(1,1)\in \C^2$ is in $\cD_F$, but $-(1,1)\notin\cD_F$. Thus $\cD_F$
 is not rotationally invariant and hence not a spectraball. Hence 
  Theorem \ref{thm:main} can not be applied
 to bianalytic mappings $\varphi:\cD_F\to \cD_B$.
 Since  $F$ is nilpotent,  $\cD_F$ is not sv-generic (Lemma \ref{l:ballnotsv})
 and therefore Theorem  \cite[Theorem 1.8]{AHKM} can not be applied
 to bianalytic mappings $\varphi:\cD_F\to \cD_B$ (even assuming $B$ is sv-generic).
 On the other hand,
 $F$ does span an algebra and thus Proposition \ref{prop:properobvious} applies.  
  A straightforward calculation shows that the origin-preserving birational map
    $q:\cD_F\to \cB_F$ of Proposition \ref{prop:properobvious} 
 is given by $q(x_1,x_2) = (x_1, x_2 + x_1^2)$. 
Evidently $\cB_F=\cB_E$.  
The following proposition summarizes the discussion above.

\begin{prop}
 \label{prop:model}
 The mapping 
\[
 q(x_1,x_2) = (x_1, x_2+x_1^2)
\]
  is a bianalytic map from $\cD_F$ onto $\cB_E$. Further, $E$ is sv-generic,
 but $\cD_F$ is neither a spectraball nor sv-generic.
\end{prop}

According to Proposition \ref{prop:model},
 to classify bianalytic maps 
 $f:\cD_F\to \cD_B$ it suffices to   determine 
 the bianalytic maps $h:\cB_E\to \cD_B.$  
Such maps are the subject of 
the next subsection.

\subsection{Bianalytic mappings of $\cB_E$ to a free spectrahedron $\cD_B$}

Theorem \ref{thm:main} applies in the case that $B$ is sv-generic or has size $2$.

\begin{prop}
\label{prop:BZZZZ}
 Suppose $B\in M_e(\C)^2$ and either $e=2$ or $B$ is sv-generic. If $f:\cB_E\to\cD_B$ is bianalytic, then $e=2$ and there is a unimodular $\alpha$ and 
$2\times 2$ unitary $M$ such that $B= \alpha M^*EM$ and further $f$ is the
birational map
\[
f(x)
=  \begin{pmatrix} x_1(1-\alpha x_1)^{-1} & (1-\alpha x_1)^{-1} x_2 (1-\alpha x_1)^{-1} 
  \end{pmatrix}.
\]
\end{prop}

\begin{remark}\rm
 The mapping $f$ is a variant (obtained by the linear change of variable $(x_1,x_2)$ maps to $\alpha(x_1,x_2)$)
   of those appearing in $g=2$ type IV algebra (see \cite[Section 8.3]{AHKM} or Subsubsection \ref{sec:g2t4} below). 
\end{remark}

\begin{proof}[Proof of Proposition~\ref{prop:BZZZZ}]  
 In this case the $Z$ in Theorem \ref{thm:main} is a unimodular multiple
 of the identity. Indeed, by \eqref{eq:secretalg},
 $$ Z = E_1 ZE_1 = 
  (\Xi_1)_{1,1} I  +   (\Xi_1)_{1,2} E_2
 $$
 and since $Z$ is unitary, it follows that $(\Xi_1)_{1,2} =0$ and $Z=\alpha I$. 
 It is now easy to verify that $\Xi=\alpha E$.
 Hence the corresponding convexotonic map is
\[
\begin{split}
 f(x) & =  x(I-\Lambda_\Xi(x))^{-1} \\
 & =  \begin{pmatrix} x_1 & x_2 \end{pmatrix}\,
   \begin{pmatrix} 1-\alpha x_1 & -\alpha x_2 \\ 0 & 1-\alpha x_1 \end{pmatrix}^{-1}\\
& =  \begin{pmatrix} x_1(1-\alpha x_1)^{-1} & (1-\alpha x_1)^{-1} x_2 (1-\alpha x_1)^{-1} 
  \end{pmatrix},
\end{split} 
\]
as desired.
\end{proof}

Composing the $f$ from Proposition \ref{prop:BZZZZ} with the original 
$q=(x_1,x_2+x_1^2)$, the bianalytic map
between $\cD_F$ and $\cB_E$, gives the mapping
 from the original domain $\cD_F$ to $\cD_B$,
\[
  f\circ q = \begin{pmatrix}  x_1(1-\alpha x_1)^{-1} & (1-\alpha x_1)^{-1} \left [x_2 +x_1^2\right ] (1-\alpha x_1)^{-1}  \end{pmatrix}.
\]
By  \cite[Theorem 1.8]{AHKM}, if $G$, $H$ and $K$ are all sv-generic and  $r:\cD_G\to \cD_H$ and $s:\cD_H\to \cD_K$ are bianalytic (and extend to be analytic on pseudoconvex domains
containing $\cD_G$ and $\cD_H$ respectively),
 then $r,s$ and $r\circ s$ are convexotonic. However,
generally one does not expect an arbitrary composition of convexotonic maps to be convexotonic. 
(See \cite[Subsection 8.4]{AHKM}.)    Thus,
it is of interest to note that, even though our $F$ is not sv-generic, the
map $f\circ q$ is convexotonic.

\begin{prop}
 \label{lem:isto}
  The map $f\circ q$ is convexotonic corresponding to the tuple $\Xi = \begin{pmatrix} \alpha I_2 + E_2, \alpha E_2 \end{pmatrix}$.
\end{prop}  

\begin{proof}
 Here is an outline of the computation that proves the proposition.
\[
 \begin{split}
   x(I&-\Lambda_\Xi(x))^{-1}= x \begin{pmatrix} 1-\alpha x_1 & -(x_1+\alpha x_2) \\ 0 & 1-\alpha x_1\end{pmatrix}^{-1} \\
  & =  \begin{pmatrix} x_1 & x_2 \end{pmatrix}\, 
     \begin{pmatrix} (1-\alpha x_1)^{-1} & (1-\alpha x_1)^{-1}(x_1+\alpha x_2)(1-\alpha x_1)^{-1} \\ 0 & (1-\alpha x_1)^{-1}\end{pmatrix} \\
  &= \begin{pmatrix}(x_1 (1-\alpha x_1)^{-1} &   & x_1 (1-\alpha x_1)^{-1}(x_1+\alpha x_2) (1-\alpha x_1)^{-1} +x_2 (1-\alpha x_1)^{-1} \end{pmatrix}.
 \end{split}
\]
 Analyzing the second entry above gives
\[
\begin{split}
x_1 (1&-\alpha x_1)^{-1} (x_1+\alpha x_2) (1-\alpha x_1)^{-1} +x_2 (1-\alpha x_1)^{-1} \\
 &= (1-\alpha x_1)^{-1} [ x_1^2 + \alpha x_1 x_2 + (1-\alpha x_1)x_2 ] (1-\alpha x_1)^{-1}\\
 &=  (1-\alpha x_1)^{-1} [x_2 + x_1^2] (1-\alpha x_1)^{-1},
\end{split}
\]
 as desired.
\end{proof}

\subsection{Two dimensional algebras with $g=2$}
In this section we consider, in view of Corollary \ref{cor:obvious}, the four indecomposable algebras $\cA$ of dimension two.
In each case we choose a tuple $\cR=(\cR_1,\cR_2)$  and compute the resulting convexotonic map $G:\cD_{\cR}\to \cB_{\cR}$.
We adopt the names for these algebras used in \cite{AHKM}.

\subsubsection{$g=2$ type $I$ algebra}
Let $\cR=F$, where $F$ is given by 
 \eqref{eq:F}. In this case we already saw $q(x_1,x_2)= (x_1, x_2+x_1^2)$.
In this case $\cD_F$ and $\cB_F$ are both bounded. 
While the tuple $F$ is not sv-generic, the tuple $E$ of equation \eqref{eq:E} is and moreover $\cB_F=\cB_E$.
Hence Theorem \ref{thm:main} does indeed apply (by replacing $F$ by $E$).

\subsubsection{$g=2$ type $II$ algebra}
Let 
\[
 \cR_1 =\begin{pmatrix} 1 &0\\0&0\end{pmatrix}, \ \ \ \cR_2 =\begin{pmatrix} 0&1\\0&0\end{pmatrix}.
\]
We have
\[
 (I+\Lambda_{\cR}(x))^{-1}\Lambda_\cR(x)
  = \begin{pmatrix} (1+x_1)^{-1} x_1 & (1+x_1)^{-1} x_2 \\ 0&0\end{pmatrix}.
\]

Hence  $q=((1+x_1)^{-1} x_1 \;\;\;  (1+x_1)^{-1} x_2)$ is a birational map from $\interior(\cD_{\cR})$  to the  spectraball $\interior(\cB_{\cR})$
that also maps the boundary of $\cD_{\cR}$ into the boundary of $\cB_{\cR}$. 
On the other hand, if $X_1$ is skew selfadjoint, then $(X_1,0)\in \cD_{\cR}$, so that $\cD_{\cR}$ is not bounded and, for instance, the tuple
\[
 \left (\begin{pmatrix} 0 & -1\\ 1 & 0 \end{pmatrix}, \begin{pmatrix} 0 &0\\0&0 \end{pmatrix} 
 \right )
\]
is in $\cB_{\cR}$ but not the range of $q$.
 In this example,  $\cR$ has a (common nontrivial) cokernel and is thus not sv-generic.
 Hence Theorem \ref{thm:main} does not apply.

\subsubsection{$g=2$ type $III$ algebra}
This case, in which 
\[
 \cR_1 =\begin{pmatrix} 1&0\\0&0\end{pmatrix}, \ \ \  \cR_2 = \begin{pmatrix} 0&0\\1&0\end{pmatrix},
\]
is very similar to the $g=2$ type $II$ case.

 \subsubsection{$g=2$ type $IV$ algebra}
\label{sec:g2t4}
Let $\cR=E$, where $E$ is defined in equation \eqref{eq:E}, and observe
\[
 (I+\Lambda_\cR(x))^{-1}\, \Lambda_\cR(x) 
 =  \begin{pmatrix} (1+x_1)^{-1} x_1 &  (1+x_1)^{-1} x_2(1+x_1)^{-1} \\ 0 & (1+x_1)^{-1} x_1 \end{pmatrix}.
\]
In this case,
\[
 q(x) = \begin{pmatrix} x_1(1+ x_1)^{-1} & (1+ x_1)^{-1} x_2 (1+ x_1)^{-1} \end{pmatrix}
\]
is bianalytic from $\interior(\cD_{\cR})$ to $\interior(\cB_E)$ and maps boundary into boundary, but
does not map boundary onto boundary.  In this case $\cB_{\cR}$ is bounded and sv-generic and hence
Theorem \ref{thm:main} does apply (with appropriate assumptions on $\cD_B$ and $p:\cD_{\cR}\to \cD_B$).

\end{document}